\documentclass[12pt]{amsart}
\usepackage{amssymb, amstext, amscd, amsmath}
\usepackage{mathtools, xypic, color, dsfont}

\makeatletter
\def\@cite#1#2{{\m@th\upshape\bfseries%
[{#1\if@tempswa{\m@th\upshape\mdseries, #2}\fi}]}}
\makeatother

\mathtoolsset{centercolon} 

\theoremstyle{plain}
\newtheorem{thm}{Theorem}[section]

\newtheorem{prop}[thm]{Proposition}
\newtheorem{lem}[thm]{Lemma}

\theoremstyle{definition}
\newtheorem{rem}[thm]{Remark}

\newcommand{\bC}{{\mathds{C}}}
\newcommand{\bD}{{\mathds{D}}}
\newcommand{\bN}{{\mathds{N}}}

\newcommand{\bT}{{\mathds{T}}}

  \newcommand{\A}{{\mathcal{A}}}
  
\renewcommand{\H}{{\mathcal{H}}}
  \newcommand{\K}{{\mathcal{K}}}

\newcommand{\fM}{{\mathfrak{M}}}
\newcommand{\rA}{{\mathrm{A}}}
\newcommand{\ep}{{\varepsilon}}

\newcommand{\qand}{\quad\text{and}\quad}
\newcommand{\qfor}{\quad\text{for}\ }

\newcommand{\AND}{\text{ and }}

\newcommand{\Alg}{\operatorname{Alg}}
\newcommand{\spn}{\operatorname{span}}
\newcommand{\re}{\operatorname{Re}}

\newcommand{\ip}[1]{\langle #1 \rangle}
\newcommand{\ol}{\overline}
\newcommand{\ltwo}{{\ell^2}}
\newenvironment{sbmatrix}{\left[\begin{smallmatrix}}{\end{smallmatrix}\right]}

\begin{document}

\title{A $3\times3$ dilation counterexample}

\author[M.D. Choi]{Man Duen Choi}
\address{Mathematics Department\\University of Toronto\\ Toronto, ON\;  \newline
M5S 2E4\\CANADA}
\email{choi@math.toronto.edu}

\author[K.R. Davidson]{Kenneth R. Davidson}
\address{Pure Math. Department\\University of Waterloo\\ Waterloo, ON\;
N2L 3G1\\CANADA}
\email{krdavids@uwaterloo.ca}

\begin{abstract}
We define four $3\times3$ commuting contractions which
do not dilate to commuting isometries. However they do satisfy
the scalar von Neumann inequality.
These matrices are all nilpotent of order 2.
We also show that any three $3\times3$ commuting contractions which
are scalar plus nilpotent of order 2 do dilate to commuting isometries.
\end{abstract}

\thanks{Partially supported by grants from NSERC}

\dedicatory{Dedicated to the memory of William B. Arveson}

\date{}
\maketitle

\section{Introduction}\label{S:intro}

Seminal work of Sz.Nagy \cite{SzN} showed that every contraction $A$ has a coextension
to an isometry of the form
\[ S = \begin{bmatrix}A&0\\ *&*\end{bmatrix} .\]
This provides a simple proof of von Neumann's inequality:
\[ \|p(A)\| \le \|p\|_\infty := \sup_{|z|\le1} |p(z)| \]
for all polynomials. Indeed, this remains valid for matrices of polynomials
\[
 \big\| \big[ p_{ij}(A) \big] \big\| \le 
 \big\| \big[ p_{ij} \big] \big\| := 
 \sup_{|z|\le1} \big\| \big[ p_{ij} (z)\big] \big\| .
\]
A decade later, Ando \cite{Ando} showed that two commuting contractions have 
simultaneous coextensions to a common Hilbert space which are
commuting isometries. 
This yields the 2-variable matrix von Neumann inequality.

However Varopoulos \cite{Var} showed that there exist three commuting contractions
which do not satisfy von Neumann's inequality, and therefore do not have
a simultaneous coextension to three commuting isometries.
In the appendix, he and Kaijser provide an example with $5\times5$ matrices.
A related example of Parrott \cite{Par} provides three commuting contractions
which do satisfy the scalar von Neumann inequality but fail the matrix version.
Thus they also do not dilate.
A dilation theorem of Arveson \cite{Arv1} shows that there is a simultaneous coextension
to commuting isometries if and only if the matrix von Neumann inequality holds.
A nice treatment of this material is contained in Paulsen \cite{PaulCB},
and the earlier material is contained in the classic text \cite{SF}.

Holbrook \cite{Holb1} found three $4\times4$ matrices which are commuting 
contractions yet do not coextend to commuting isometries. 
He also showed \cite{Holb2} that for $2\times2$ matrices, arbitrary commuting
families of commuting contractions do have commuting isometric coextensions.
See also \cite{Drury}.
Holbrook asks what the situation is for $3\times3$ matrices.

Sometimes these results are stated instead for unitary (power) dilations
of the form
\[ U_i  = \begin{bmatrix} *&0&0\\ *&A_i&0\\ *&*&*\end{bmatrix} \]
on a Hilbert space $\K = \K_- \oplus \H \oplus \K_+$.
If these unitaries commute, then the
restriction of $U_i$ to $\H \oplus \K_+$, namely the lower $2\times2$ corner,
yields commuting isometric  coextensions $S_i$; (and the compression to $\K_- \oplus \H$ 
yields commuting coisometric extensions of the $A_i$).
 
Conversely, suppose that commuting contractions $A_i$ have coextensions to commuting
isometries $S_i$.
An old result of Ito and Brehmer \cite[Proposition I.6.2]{SF} shows that 
the $S_i$ dilate to commuting unitaries $U_i$ so that
\[
 P_\H U_1^{k_1} \cdots U_s^{k_s}|_\H = A_1^{k_1} \cdots A_s^{k_s} 
 \qfor k_i \ge 0. 
\]
It follows that they simultaneously have a triangular form as given above.
Therefore these two formulations are equivalent.

In this note, we provide an example of four commuting contractions in the  
$3\times3$ matrices $\fM_3$
which cannot be coextended to commuting isometries. The question of whether
there are three commuting $3\times3$ contractive matrices
which cannot be coextended to commuting isometries remains open.

Our examples are nilpotents of order 2.
We show that any finite family of commuting contractions of this form
always satisfy the scalar von Neumann's inequality.
It would be of interest to know whether the scalar von Neumann inequality
holds for commuting $3\times3$ contractions.
For our counterexample, we exhibit a specific matrix polynomial which shows that 
the matrix valued von Neumann inequality fails.

We also show that any three commuting $3\times 3$ contractions which are 
of the form scalar plus nilpotent of order 2 always do have commuting
isometric coextensions.

When the second author discovered these examples, he found out that 
the first author and Y. Zhong had a similar example 
from many years ago which was never published. 
Zhong has left academia, and could not be contacted---but he shares in the
credit for this work.

\section{The Example}\label{S:ex}

Let $\H=\bC^3$ have an orthonormal basis $f,e_1,e_2$. 
For $i = 1,2$, pick $\theta_i \in (0,\pi/2)$ and set 
$c_i = \cos\theta_i$ and $s_i = \sin\theta_i$.
Define 
\begin{alignat*}{3}
A_1 &= e_1f^* &\qquad\qquad A_2 &= e_2f^*\\ 
&= \begin{bmatrix}0&0&0\\1&0&0\\0&0&0\end{bmatrix}
&&= \begin{bmatrix}0&0&0\\0&0&0\\1&0&0\end{bmatrix} 
\shortintertext{and}
A_3 &= (c_1e_1+s_1e_2) f^* &\qquad A_4 &= (c_2e_1 +is_2e_2) f^* \\
&= \begin{bmatrix}0&0&0\\c_1&0&0\\s_1&0&0\end{bmatrix}
&&= \begin{bmatrix}0&0&0\\c_2&0&0\\is_2&0&0\end{bmatrix}.
\end{alignat*}
Observe that $A_iA_j = 0$ for all $1 \le i,j \le 4$.

\begin{thm} \label{T:no_dilation}
The matrices $A_1$, $A_2$, $A_3$ and $A_4$ do not have simultaneous coextensions
to commuting isometries.
\end{thm}

\begin{proof}
Suppose that the $A_i$ coextend to commuting isometries 
\[ S_i = \begin{bmatrix}A_i & 0 \\ D_i & V_i\end{bmatrix}  \qfor i = 1,2,3,4\]
on a Hilbert space $\K$. 
We may assume that $\K$ is the smallest invariant subspace for the $S_i$
containing $\H$.  That is, if we write $S^k = S_1^{k_1} S_2^{k_2}S_3^{k_3}S_4^{k_4}$
for $k = (k_1,k_2,k_3,k_4) \in \bN^4$, then
\[ \K = \bigvee_{k\in\bN^4} S^k \H. \]

Observe that since $\|A_i f\|=1$, we have
\[ S_i f = A_i f \qfor i = 1,2,3,4. \]
Hence
\[ (c_1S_1 + s_1S_2 - S_3) f = 0 .\]
Therefore
\[ 0 = S^k (c_1S_1 + s_1S_2 - S_3) f = (c_1S_1 + s_1S_2 - S_3) S^k f .\]
Since $S_i f = e_i$ for $i=1,2$, 
\[ \ker (c_1S_1+s_1S_2-  S_3) \supset \bigvee_{k\in\bN^4} S^k \spn\{f,S_1f,S_2f\} = \K .\]
So 
\[ S_3 = c_1S_1 + s_1S_2 .\]
Similarly, 
\[ S_4 = c_2S_1 +is_2S_2 .\]

Since $S_3$ is an isometry, 
\begin{align*}
 I &= S_3^* S_3 \\
 &= c_1^2 S_1^* S_1 + s_1^2 S_2^* S_2 + c_1s_1 (S_2^*S_1 + S_1^*S_2) \\
 &= I +  c_1s_1 (S_2^*S_1 + S_1^*S_2) .
\end{align*}
It follows that 
\[ S_2^*S_1 + S_1^*S_2 = 0 .\]
Likewise, using the fact that $S_4$ is an isometry,
\[ 0 = (iS_2)^*S_1 + S_1^*(iS_2) = -i(S_2^*S_1 - S_1^*S_2). \]
Therefore 
\[ S_2^*S_1 = 0 .\]
This implies that $S_1$ and $S_2$ have pairwise orthogonal ranges,
and therefore do not commute.
This contradiction establishes the result.
\end{proof}

\begin{rem}
It is possible to coextend any three of these operators to commuting isometries.
This will follow from Theorem~\ref{T:nilpotent}.
Indeed, the construction given there will simultaneously coextend $A_1$, $A_2$
and all matrices $A_j = (c_je_1+s_je_2)f^*$ such that $c_j\bar s_j$ have a common
argument. So this example could not be given using real matrices.
\end{rem}

\begin{rem}
It is not even possible to find commuting isometries of the form
\[ S_i = \begin{bmatrix} A_i & B_i\\ C_i&D_i \end{bmatrix} .\]
As in our proof above, we may suppose that $\H$ is a cyclic subspace.
Exactly the same argument shows that 
\[
 S_3 = c_1S_1 + s_1S_2 \qand
 S_4 = c_2S_1 +is_2S_2 .
\]
Hence  $S_2^*S_1 = 0$, which contradicts commutativity.
\end{rem}

\section{Von Neumann's Inequality} \label{S:vN}

It is also of interest to find commuting contractions in $\fM_3$ that fail the
scalar von Neumann inequality. Apparently computer searches for such
examples have been unsuccessful. We show that our example does satisfy
the scalar von Neumann inequality. So it does not settle this question.
We provide an explicit matrix polynomial which fails the matrix von Neumann inequality.

\begin{lem} \label{L:2nilpotent}
If $A_1,\dots,A_n \in \fM_3$ are commuting nilpotents of order 2, then
there is a vector $f$ and vectors $v_i \in (\bC f)^\perp$ so that either
$A_i = v_i f^*$ for $1 \le i \le n$ or $A_i = fv_i^*$ for $1 \le i \le n$.
\end{lem}

\begin{proof}
Observe that a nilpotent $A \in \fM_3$ of order 2 must have rank one.
Thus it can be expressed as $A = vf^*$ where 
$f$ is a unit vector. And since 
\[ 0 = A^2 = \ip{v,f}A ,\]
we have that $\ip{v,f}=0$.
If two such non-zero operators $A_i = v_if_i^*$ commute, then
\[
 \ip{v_2,f_1} v_1 f_2^* = (v_1 f_1^*)(v_2 f_2^*) 
 = (v_2 f_2^*) (v_1 f_1^*) = \ip{v_1,f_2} v_2 f_1^* .
\]
So either 
\[ \ip{v_2,f_1} = \ip{v_1,f_2} = 0 \]
or $v_2f_2^*$ is a multiple of $v_1f_1^*$,
in which case this identity remains true.
So 
\[ \spn\{ v_1,v_2 \} \perp \spn\{f_1,f_2\} .\]

Furthermore, if $n$ non-zero operators $v_if_i^*$ commute, 
then the pairwise relations yield
\[ \spn\{v_1,\dots,v_n\} \perp \spn\{f_1,\dots,f_n\} .\]
Therefore one of these subspaces is 1-dimensional.
By taking adjoints if necessary,
we may suppose that $\spn\{f_1,\dots,f_n\}$ is one dimensional.
So after a scalar change, we have $A_i = v_i f^*$ for $i = 1,\dots,n$;
and each $v_i$ belongs to $(\bC f)^\perp$.
\end{proof}

\begin{prop} \label{P:vN}
Any finite number of commuting contractions $A_1$, \dots, $A_n$
in $\fM_3$ of the form scalar plus order 2 nilpotent  
satisfy the scalar von Neumann inequality.
\end{prop}

\begin{proof}
We first suppose that each $A_i$ is a nilpotent of order 2.
By Lemma~\ref{L:2nilpotent}, we may suppose that there is a unit vector $f$
and vectors $v_i \in (\bC f)^\perp$ of norm at most 1 so that $A_i = v_i f^*$.
Since $A_iA_j = 0$ for all $1 \le i,j \le n$, we need concern ourselves only
with the linear part of a polynomial.

Consider a polynomial
\[ p(z) = c + \sum_{i=1}^n a_i z_i + q(z) \]
where $q(z)$ consists of higher order terms.
Without loss of generality, we may suppose that some $a_i \ne 0$.
Let $\lambda_i$ be scalars of modulus 1 so that $a_i = \lambda_i |a_i|$.
Set 
\[ B = \Big( \sum_{i=1}^n |a_i| \Big)^{-1} \sum_{i=1}^n a_i A_i . \]
Clearly $\|B\| \le 1$.  Then
\begin{align*}
 p(A_1,\dots,A_n) &= cI + \sum_{i=1}^n a_i A_i 
 = cI + \sum_{i=1}^n |a_i| B \\&
 =  p( \ol{\lambda}_1 B, \dots, \ol{\lambda}_n B) = q(B) ,
\end{align*}
where $q(x) = p(\ol{\lambda}_1 x, \dots, \ol{\lambda}_n x)$.
Hence by von Neumann's inequality, 
\[
  \| p(A_1,\dots,A_n) \| = \|q(B)\| \le \|q\|_\infty \le \|p\|_\infty .
\]

Now suppose that $A_i = \lambda_i I + N_i$ where $N_i^2=0$.
If $N_i\ne 0$, the fact that $\|A_i\|\le1$ implies that $|\lambda_i|<1$.
If some $A_i = \lambda_i I$ with $|\lambda_i|=1$, replace it with
$A'_i = (1-\ep)\lambda_i I$ for $\ep>0$ small. If we establish the 
inequality for this new $n$-tuple, we recover the case we desire
by letting $\ep$ tend to 0.

Define M\"obius maps 
\[ b_i(z) = \frac{z-\lambda_i}{1-\ol{\lambda_i}z} \qfor 1 \le i \le n.\]
Then
\[ B_i :=b_i(A_i) = (1-|\lambda_i|^2)^{-1} N_i \]
are commuting  nilpotents of order 2, and  $A_i = b_i^{-1}(B_i)$. 
Moreover, $B_i$ are contractions by the one variable von Neumann
inequality (or by direct computation).
If $p$ is a polynomial, then 
\[ q(z) = p(b_1^{-1}(z_1),\dots,b_n^{-1}(z_n)) \]
is a function in the polydisc algebra $\rA(\bD^n)$ of the same norm
as $p$ because each $b_i^{-1}$ takes the circle $\bT$ onto itself.
Hence 
\[ \| p(A_1,\dots,A_n) \| = \| q(B_1,\dots,B_n) \| \le \|q\|_\infty = \|p\|_\infty .\]
Therefore the scalar von Neumann inequality is satisfied.
\end{proof}

If  the matrix von Neumann inequality holds for a 4-tuple of matrices $A_1$,\dots,$A_4$,
then the canonical map from $\rA(\bD^4)$ into $\A = \Alg(\{A_i\})$ taking
$z_i$ to $A_i$ for $1 \le i \le 4$ is completely contractive. 
Therefore Arveson's Dilation Theorem \cite{Arv1} shows that the 4-tuple
does dilate to commuting unitaries.  Hence the restriction to the common invariant 
subspace generated by the original space $\H$ yields a coextension to
commuting isometries. Hence there is a matrix polynomial which shows
that this inequality fails for our example in Theorem~\ref{T:no_dilation}. 
We will exhibit one.

Recall the notation of Theorem~\ref{T:no_dilation}. 
Apply the Gram-Schmidt process to the vectors $u_1$ and $u_2$ to get 
orthogonal unit vectors $f_1$ and $f_2$, where
\[ 
 u_1 = \begin{bmatrix}1\\0\\c_1\\c_2\end{bmatrix} ,\quad
 u_2 = \begin{bmatrix}0\\1\\s_1\\is_2\end{bmatrix} ,\quad
 f_1 = \begin{bmatrix}\alpha_1\\\alpha_2\\\alpha_3\\\alpha_4\end{bmatrix} ,\qand
 f_2 = \begin{bmatrix}\beta_1\\\beta_2\\\beta_3\\\beta_4\end{bmatrix} .
\]

\begin{prop} \label{P:fail_vN}
Let 
\[
 p(z) = \begin{bmatrix}
 \bar\alpha_1z_1 + \bar\alpha_2z_2 + \bar\alpha_3z_3 + \bar\alpha_4z_4\\
 \bar\beta_1z_1 + \bar\beta_2z_2 + \bar\beta_3z_3 + \bar\beta_4z_4
 \end{bmatrix} .
\]
Then 
\[ \|p\|_\infty = \sup_{|z_i|=1}\|p(z)\| < 2 = \|p(A_1,A_2,A_3,A_4)\|.\]
So the matrix von Neumann inequality fails for $(A_1, A_2, A_3, A_4)$.
\end{prop}

\begin{proof}
Let $z=(z_1,z_2,z_3,z_4)^t$ with $|z_i| = 1$. 
Since $f_1$ and $f_2$ are orthonormal,  
\[ \|p(z)\|^2 = |\ip{z,f_1}|^2+|\ip{z,f_2}|^2 \le \|z\|^2 = 2 .\]
This inequality is strict unless $z \in \spn\{f_1,f_2\} = \spn\{u_1,u_2\}$. 
By compactness, the norm of $p$ is exactly 2 only if this value is obtained.
However we claim that no $z$ with $|z_i|=1$ lies in this subspace.
Indeed, suppose that
\[
 z = \begin{bmatrix}z_1\\z_2\\z_3\\z_4\end{bmatrix} =
 a\begin{bmatrix}1\\0\\c_1\\c_2\end{bmatrix} +
 b\begin{bmatrix}0\\1\\s_1\\is_2\end{bmatrix} =
 \begin{bmatrix}a\\b\\ac_1+bs_1\\ac_2+ibs_2\end{bmatrix}
\]
We require
\[ 1 = |a| = |b| = |ac_1+bs_1| = |ac_2+ibs_2| .\]
Arguing as in the proof of Theorem~\ref{T:no_dilation}, we find that
$\re (\bar a b) = 0 = \re(i\bar a b)$. Thus $\bar a b = 0$, contradicting $|a|=|b|=1$.
So $\|p\|<2$.

Now we compute $\|p(A_1,A_2,A_3,A_4)\|$. 
Clearly it suffices to consider the $2,1$ and $3,1$ entries. That is
\begin{align*}
 p\big( \begin{bmatrix}1\\0\end{bmatrix}, 
 \begin{bmatrix}0\\1\end{bmatrix},\begin{bmatrix}c_1\\s_1\end{bmatrix},
 \begin{bmatrix}c_2\\is_2\end{bmatrix} \big) &=
 \begin{bmatrix}\bar\alpha_1 +  \bar\alpha_3c_1 + \bar\alpha_4c_2\\
 \bar\alpha_2 + \bar\alpha_3s_1 + \bar\alpha_4is_2\\
 \bar\beta_1 + \bar\beta_3c_1 + \bar\beta_4c_2\\
 \bar\beta_2 + \bar\beta_3s_1 + \bar\beta_4 is_2\end{bmatrix}  =
 \begin{bmatrix}\ip{u_1,f_1}\\ \ip{u_2,f_1}\\ \ip{u_1,f_2}\\ \ip{u_2,f_2} \end{bmatrix}
\end{align*}
Therefore
\[ 
 \|p(A_1,A_2,A_3,A_4)\|^2 = \sum_{i=1}^2 \sum_{j=1}^2 |\ip{u_i,f_j}|^2 = 
 \sum_{i=1}^2 \|u_i\|^2 = 4. \qedhere
\]
\end{proof}

\section{Dilating Three $3\times3$ Nilpotents of Order 2}

The purpose of this section is to prove a positive dilation result
for certain triples of $3\times3$ commuting contractions.
First we need a couple of lemmas.

\begin{lem} \label{L:scale}
Suppose that three commuting contractions $A_1$, $A_2$ and $A_3$
have coextensions $S_1$, $S_2$ and $S_3$ which are commuting isometries.
Then if $|a_i| \le 1$ for $i=1,2,3$, the operators $a_iA_i$ also 
have coextensions which are commuting isometries.
\end{lem}

\begin{proof}
Observe that $a_iA_i$ coextend to commuting contractions $a_iS_i$.
If $|a_1|<1$, let $d_i = (1-|a_i|^2)^{1/2}$ and coextend $a_1S_1$ to
\[
 T_1 = \begin{bmatrix} 
 a_1S_1 & 0 & 0 & 0 & \dots\\
 d_1S_1 & 0 & 0 & 0 & \dots\\
 0 & S_1 & 0 & 0 & \dots\\
 0 & 0 & S_1 & 0 & \dots\\
 \vdots&\vdots&\ddots&\ddots&\ddots
\end{bmatrix}
\]
and simultaneously for $i=2,3$, coextend $a_iS_i$ to 
\[
 a_i T_i = a_i S_i \otimes I = 
 \begin{bmatrix} 
 a_1S_i & 0 & 0 & 0 & \dots\\
 0 & a_1S_i & 0 & 0 & \dots\\
 0 & 0 & a_1S_i & 0 & \dots\\
 0 & 0 & 0 & a_1S_i & \dots\\
 \vdots&\vdots&\ddots&\ddots&\ddots
 \end{bmatrix}
\]
It is easy to see that these contractions commute, and $T_i$ are isometries.

Now dilate $a_2T_2$ to an isometry, and dilate $T_1$ and $T_3$ to commuting
isometries in the same manner. 
Finally, repeat a third time to dilate the third term to an isometry.
\end{proof}

\begin{lem} \label{L:3vectors}
Given three unit vectors $v_1$, $v_2$, $v_3$ in $\bC^2$,
there exist three commuting unitaries $U_1$, $U_2$, $U_3$ in $\fM_2$ such that
\[ U_i v_j = U_j v_i \qfor 1 \le i,j \le 3. \]
\end{lem}

\begin{proof}
We may choose an orthonormal basis for $\bC^2$
in which $v_1 = \begin{sbmatrix}0\\1\end{sbmatrix}$.  We set $U_1=I_2$.
If $v_2 = e^{i\theta}v_1$, let $U_2=e^{i\theta}I_2$ and choose $U_3$ to be
any unitary matrix such that $U_3v_1=v_3$. This works.

Otherwise, $v_1$ and $v_2$ are linearly independent.
Write $v_3 = av_1+bv_2$.
For convenience, we may multiply $v_3$ and $v_2$ by scalars of modulus 1 
so that $a$ and $b$ are real. This does not affect the problem.
Write
\[
 v_2 = \begin{bmatrix}\alpha_2\\ \beta_2\end{bmatrix} ,\quad
 v_3 = \begin{bmatrix}\alpha_3\\ \beta_3\end{bmatrix} ,\quad
 U_2 = \begin{bmatrix}\phantom{-}\bar\beta_2 & \alpha_2 \\
            -\bar\alpha_2 &\beta_2 \end{bmatrix} \AND
 U_3 = \begin{bmatrix}\phantom{-}\bar\beta_3 &\alpha_3 \\
            -\bar\alpha_3 &\beta_3 \end{bmatrix} .
\]
A simple calculation shows that
\[ 
 a U_1 + b U_2 = 
 \begin{bmatrix}a + b\bar\beta_2 & b\alpha_2\\
 -b\bar\alpha_2 &  a + b\beta_2 \end{bmatrix} =
 \begin{bmatrix}\phantom{-}\bar\beta_3 &\alpha_3 \\
            -\bar\alpha_3 &\beta_3 \end{bmatrix} = U_3 .
\]
This shows that the unitary matrices commute.
By construction, 
\[ U_iv_1 = v_i = U_1 v_i \qfor i=2,3 .\]
Finally observe that 
\[ 
 U_3 v_2 = (a U_1 + b U_2)v_2 = av_2 + bU_2v_2 = U_2(av_1+bv_2) = U_2v_3 .
 \qedhere
\]
\end{proof}

\begin{thm} \label{T:nilpotent}
Suppose that $A_1$, $A_2$ and $A_3$ are three commuting $3\times3$
matrix contractions which are all of the form scalar plus nilpotent of order 2.
Then there exist commuting isometric coextensions $S_1$, $S_2$ and $S_3$
of $A_1$, $A_2$ and $A_3$.
\end{thm}

\begin{proof}
First assume that the $A_i$ are nilpotent of the form $A_i = v_i f^*$ for $i = 1,2,3$,
where the $v_i$ are \textit{unit vectors} in $(\bC f)^\perp\simeq \bC^2$.
Choose a basis for $\bC^2$ so that $v_1 = \begin{sbmatrix}0\\1\end{sbmatrix}$.  
Let $U_i$ be the commuting $2\times 2$ unitaries given by Lemma~\ref{L:3vectors}.
Then the construction yields the identities $U_iv_1 = v_i$ for $i=1,2,3$.
So the second column of each $U_i$ is $v_i$, say 
$U_i = \begin{bmatrix} \gamma_i&\alpha_i\\ \delta_i&\beta_i\end{bmatrix}$. 

Observe that if $U$ is the bilateral shift on $\ltwo$, then $W_i = U \otimes U_i$ 
are commuting unitaries of the form:
\[
 W_i = 
  \left[\begin{array}{ccc|cc|ccc}
  \ddots&\ddots&\ddots&\vdots&\vdots&\vdots&\vdots& \\
  \ddots&O&O&O&O&O&O&\dots \\ 
  \ddots&U_i&O&O&O&O&O&\dots \\ \hline
  \dots&O&U_i&O&O&O&O&\ldots\\
  \dots&O&O&U_i&O&O&O&\dots\\ \hline
  \dots&O&O&O&U_i&O&O&\ddots\\
  \dots&O&O&O&O&U_i&O&\ddots\\
  &\vdots&\vdots&\vdots&\vdots&\ddots&\ddots&\ddots
 \end{array}\right]
\]
on $\K = \K_- \oplus \bC^4 \oplus \K_+$. Here $O$ is a $2\times2$ zero matrix.
Decompose the central $4\times4$ block of $W_i$ as $\bC \oplus \H$:
\[
 \begin{bmatrix}O&O\\U_i&O\end{bmatrix} = 
 \left[\begin{array}{c|ccc}
 0&0&0&0\\ \hline 
 0&0&0&0\\
 \gamma_i&\alpha_i&0&0\\
 \delta_i&\beta_i&0&0
 \end{array}\right] =
 \begin{bmatrix}
 0&0_3^*\\
 u_i&A_i
 \end{bmatrix}
\]
The spaces $\K_+$ and $\H \oplus \K_+$ are invariant for each $W_i$.
Therefore $W_i$ are commuting unitary (power) dilations of the $A_i$.
Note that $W_i^*$ then form commuting unitary dilations of the adjoints, $A_i^*$; 
and $\H$ is the difference of invariant subspaces $\K_-\oplus\bC^4$ and 
$\K_-\oplus\bC$.

By Lemma~\ref{L:2nilpotent}, if $A_i$ are commuting $3\times3$ nilpotents of order 2,
then either they have the form $A_i = v_if^*$  or their adjoints do. \vspace{.2ex}
By Lemma~\ref{L:scale}, it suffices to dilate the normalized matrices
$\tilde A_i := A_i/\|A_i\|$. Hence we can assume that each $\|v_i\|=1$.
The argument above produces commuting unitary dilations.
The restriction of these unitaries to the smallest common invariant subspace
containing the original 3-dimensional space yields commuting isometric
coextensions of the $\tilde A_i$.

We reduce the general case to the nilpotent case as in the previous section.
Suppose that $A_i = \lambda_i I + N_i$ where $N_i^2=0$.
Then  \[ |\lambda_i| \le \|A_i\| \le 1 .  \]
Moreover, if $|\lambda_i|=1$, then $A_i = \lambda_i I$ is already an
isometry, and we coextend it to $\lambda_i I$ on the larger space.
When $|\lambda_i|<1$, define the M\"obius map 
\[ b_i(z) = \frac{z-\lambda_i}{1-\ol{\lambda_i}z}.\]
Then $b_i(A_i)$ are commuting nilpotents of order 2.
Dilate them to commuting unitaries $W_i$ as above. 
Then define $U_i = b_i^{-1}(W_i)$.
These are commuting unitaries dilating $A_i$.
\end{proof}


\end{document}